\documentclass[a4paper,12pt,reqno]{amsart}

\usepackage{amsmath}
\usepackage{amssymb}
\usepackage{amsfonts}
\usepackage{graphicx}
\usepackage[colorlinks]{hyperref}
\renewcommand\eqref[1]{(\ref{#1})} %Need with hyperref

\graphicspath{ {images/} }
\title[Principal frequency of $p$-sub-Laplacians]{Principal frequency of $p$-sub-Laplacians for general vector fields}
\author[Michael Ruzhansky]{Michael Ruzhansky}
\address{\href{www.ruzhansky.org}{Michael Ruzhansky:}
	\endgraf
	Department of Mathematics: Analysis, Logic and Discrete Mathematics
	\endgraf
	Ghent University, Belgium
	\endgraf
	and
	\endgraf
	School of Mathematical Sciences
		\endgraf Queen Mary University of London 
			\endgraf
		United Kingdom
			\endgraf
	{\it E-mail address} {\rm Michael.Ruzhansky@ugent.be}
}

\author[Bolys Sabitbek]{Bolys Sabitbek}
\address{ \href{https://www.researchgate.net/profile/Bolys_Sabitbek3}{Bolys Sabitbek:}
		\endgraf
	School of Mathematical Sciences
	\endgraf Queen Mary University of London 
	\endgraf
	United Kingdom,
	\endgraf
	and
	\endgraf 
	Department of Mechanics and Mathematics
	\endgraf 
	Al-Farabi Kazakh National University 
	\endgraf
	71 al-Farabi Ave., Almaty, 050040 
	\endgraf Kazakhstan
	\endgraf
	and
		\endgraf
	Institute of Mathematics and Mathematical Modeling 
	\endgraf
	125 Pushkin Street., Almaty, 050010
	\endgraf
	Kazakhstan
	\endgraf
	{\it E-mail address} {\rm b.sabitbek@qmul.ac.uk}
}

\author[Durvudkhan Suragan]{Durvudkhan Suragan}
\address{\href{https://sst.nu.edu.kz/en/durvudkhan-suragan-phd/}{Durvudkhan Suragan:}
	\endgraf
	Department of Mathematics
	\endgraf
	Nazarbayev University
	\endgraf
	53 Kabanbay batyr Ave., Astana, 010000
	\endgraf
	Kazakhstan
	\endgraf
	{\it E-mail address} {\rm durvudkhan.suragan@nu.edu.kz}
}


\thanks{The first
 author was supported by FWO Odysseus 1 grant G.0H94.18N: Analysis and Partial Differential Equations and EPSRC grant EP/R003025/1. The second author was supported in parts by the MES RK grant AP08053051 and EPSRC grant EP/R003025/1. The third author was supported by the Nazarbayev University program 091019CRP2120 and the Nazarbayev University grant 240919FD3901. }

\newtheoremstyle{theorem}%name
{10pt}          % space above
{10pt}  % space below
{\sl}  % bofy font
{\parindent}     % ident - empty=no indent,  \parindent= paragraph indent
{\bf}  % thm head font
{. }    % punctuation after thm head
{ }    % space after thm head: `` ``=normal \newline=linebreak
{}     % thm head specification
\theoremstyle{theorem}

\numberwithin{equation}{section}
\theoremstyle{plain}
\newtheorem{thm}{Theorem}[section]

\newtheorem{cor}[thm]{Corollary}
\newtheorem{lem}[thm]{Lemma}

\theoremstyle{definition}

\newtheoremstyle{defi}%name
{10pt}          % space above
{10pt}  % space below
{\rm}  % bofy font
{\parindent}     % ident - empty=no indent,  \parindent= paragraph indent
{\bf}  % thm head font
{. }    % punctuation after thm head
{ }    % space after thm head: `` ``=normal \newline=linebreak
{}     % thm head specification
\theoremstyle{defi}

\begin{document}
		\begin{abstract}
		In this paper, we prove the uniqueness and simplicity of the principal frequency (or the first eigenvalue) of the Dirichlet $p$-sub-Laplacian for general vector fields. As a byproduct, we establish the Caccioppoli inequalities and also discuss the particular cases on the Grushin plane and on the Heisenberg group.
	\end{abstract}
	\maketitle
\section{Introduction}

Let $M$ be a smooth manifold of dimension $n$ with a volume form $dx$. Let $\{X_k\}_{k=1}^N$ with $n\geq N$ be a family of vector fields defined on $M$. We consider the operator
\begin{equation}
	\mathcal{L} :=\sum_{k=1}^{N}X^*_kX_k.
\end{equation}
It is is locally hypoelliptic if the commutators of the vector fields $\{X_k\}_{k=1}^N$ generate the tangent space of $M$ as the  Lie algebra. 

In addition, we define the $p$-sub-Laplacian for general vector fields by the formula
\begin{equation}\label{Lp}
	\mathcal{L}_p f := \nabla_X^* \cdot \left( |\nabla_X f|^{p-2}\nabla_X f\right), \, 1<p<\infty,
\end{equation} 
and the horizontal gradients
\begin{equation*}
	\nabla_X :=(X_1,\ldots,X_N)\,\,\, \text{and} \,\,\, \nabla_{X}^*:=(X_1^*,\ldots,X_N^*),
\end{equation*}
where 
\begin{equation*}
	X_k = \sum_{j=1}^{n}a_{kj}(x)\frac{\partial}{\partial x_j},\,\, k=1,\ldots, N,
\end{equation*}
with the formal adjoint
\begin{equation*}
X_k^* = -\sum_{j=1}^{n}\frac{\partial}{\partial x_j}(a_{kj}(x)),\,\, k=1,\ldots, N.
\end{equation*}
Let $\Omega \subset M$ be an open set. We define the functional spaces
\begin{equation}
	S^{1,p}(\Omega) =\{ u: \Omega \rightarrow \mathbb{R}; u, |\nabla_X u| \in L^p(\Omega) \}.
\end{equation}
We also consider the following functional 
\begin{equation}
	J_p(u) = \left( \int_{\Omega} |\nabla_X u|^p dx \right)^{\frac{1}{p}}.
\end{equation}
Now let us define the functional class $\mathring{S}^{1,p}(\Omega)$ to be the completion of $C_0^1(\Omega)$ in the norm generated by $J_p$, see in \cite{CDG1993}.

We consider the following Dirichlet boundary value problem for $\mathcal{L}_p$:
\begin{equation}\label{Dirichlet_BV}
\begin{cases}
	\mathcal{L}_p u = \lambda |u|^{p-2}u, \,\,& \text{in} \,\, \Omega,	\\
u=0,\,\, & \text{on} \,\, \partial \Omega. 
\end{cases}
\end{equation}
In the Euclidean setting, the first eigenvalue for the Dirichlet boundary value problem for $\mathcal{L}_p$ was obtained by Lindqvist in \cite{Lindqvist}. 
In the sub-elliptic setting, the study of boundary value problems started by the pioneering works of Bony \cite{Bony}, Gaveau \cite{Gaveau}, Kohn and Nirenberg \cite{Kohn-Nirenberg}. In 1981, the Dirichlet problem for the Kohn Laplacian on the Heisenberg group was studied by Jerison in \cite{Jer_1, Jer_2} as a part of his dissertation under the supervision of E. Stein. Also,  semilinear equations on the Heisenberg group for sums of vector fields were studied in \cite{YZP05} and \cite{Niu99}. Those works attracted significant attention to boundary value problems for the sub-elliptic operators, see e.g. \cite{CG98}, \cite{CGN02}, \cite{Danielli95}, \cite{DGMN}, \cite{GV00} and  references therein.
 
As usual, a weak solution of equation \eqref{Dirichlet_BV} means a function $u \in \mathring{S}^{1,p}(\Omega)$ such that 
\begin{equation}\label{eq_weak_sol}
	\int_{\Omega}  |\nabla_X u|^{p-2} \langle \nabla_X u, \nabla_X \phi \rangle dx - \lambda \int_{\Omega}|u|^{p-2}u \phi dx = 0,
\end{equation}
for all $\phi \in \mathring{S}^{1,p}(\Omega)$.
Similarly, by the sup-solution and sub-solution  of equation \eqref{Dirichlet_BV} we mean a function $u \in \mathring{S}^{1,p}(\Omega)$ such that 
\begin{equation}\label{inq_sup}
	\int_{\Omega}  |\nabla_X u|^{p-2}\langle \nabla_X u, \nabla_X \phi \rangle  dx - \lambda \int_{\Omega}|u|^{p-2}u \phi dx \geq 0,
\end{equation}
and
\begin{equation}\label{inq_sub}
	\int_{\Omega}  |\nabla_X u|^{p-2}\langle \nabla_X u, \nabla_X \phi \rangle dx - \lambda \int_{\Omega}|u|^{p-2}u \phi dx \leq 0,
\end{equation} 
for all $\phi \in \mathring{S}^{1,p}(\Omega)$, respectively.
Here $\lambda \in \mathbb{R}$ and $u$ are called a Dirichlet eigenvalue and eigenfunction of the operator $\mathcal{L}_p$ in $\Omega \subset M$, respectively. If we put $u$ instead of $\phi$ in \eqref{eq_weak_sol},  we have 
\begin{equation}
	\int_{\Omega} |\nabla_Xu|^p dx = \lambda \int_{\Omega} |u|^p dx,
\end{equation}
which can be written for $u \neq 0$ as 
\begin{equation*}
	\lambda := \frac{\int_{\Omega} |\nabla_Xu|^p dx}{\int_{\Omega} |u|^p dx}.
\end{equation*}
Thus, one can define the principal Dirichlet frequency $\lambda_1(\Omega)$ of the operator $\mathcal{L}_p$ as the minimum of the Rayleigh quotient $\int_{\Omega} |\nabla_Xu|^p dx / \int_{\Omega} |u|^p dx$ taken over all nontrivial functions $u \in \mathring{S}^{1,p}(\Omega)$, that is,  
\begin{equation}\label{lambda_1}
	\lambda_1(\Omega):= \min_{\mathring{S}^{1,p}(\Omega)} \left\{ \int_{\Omega} |\nabla_X u|^p dx : \int_{\Omega} |u|^p dx =1\right \}.
\end{equation}

The main aim of this paper is to prove uniqueness, simplicity, and a domain monotonicity of the principal frequency (or the first eigenvalue) of the Dirichlet $p$-sub-Laplacian for general vector fields \eqref{Lp}.
We also present the Caccioppoli inequality for this operator in the form 
\begin{equation*}
	\int_{\Omega} \phi^p|\nabla_X v|^p dx \leq p^p \int_{\Omega} v^p |\nabla_X \phi|^p dx,
\end{equation*}
for every $0\leq \phi \in C_0^{\infty}(\Omega)$, where $v>0$ is a sub-solution of \eqref{Dirichlet_BV} in $\Omega \subset M$.

Picone's identity for general vector fields plays an important role in some of our computations.
Recall Picone's identity for differentiable functions $u$ and $v$ with $v(x)\neq 0$ in $\Omega \subset \mathbb{R}^n$:
\begin{equation}\label{PiconeL_2}
	|\nabla u|^2 - \left\langle \nabla v, \nabla \left( \frac{|u|^2}{v}\right) \right\rangle = |\nabla u|^2 + \frac{|u|^2}{|v|^2} |\nabla v|^2 - 2 \frac{u}{v} \langle \nabla v, \nabla u\rangle \geq 0,
\end{equation}
where $\nabla$ and $\langle \cdot, \cdot\rangle$ are the standard gradient and the inner product in $\mathbb{R}^n$, respectively.
The Picone identity is one of the important tools in the  theory of partial differential equations (see, e.g. \cite{Swan}). As consequences of Picone's identity, one can obtain, for instance, the simplicity of principal eigenvalues, Barta's inequalities, nonexistence of positive solutions, Hardy's inequalities, and Sturmian comparison. In the Euclidean setting, the Picone identity was investigated by many authors and generalised in different directions. For example, in \cite{AH98} Allegato and Huang extended the Picone identity to the case of any $p>1$.  Recently, in \cite{Jaros_A-horm} and \cite{Jaros_Cacci} Jaros  presented the Picone identity for the Finsler $p$-Laplacian and obtained the Caccioppoli inequality, which has the form 
\begin{equation}\label{eq_Cacci_L2}
\int_{\Omega} |\phi \nabla v|^2 dx \leq 4 \int_{\Omega} |v \nabla \phi|^2 dx,
\end{equation} 
where $v>0$ is a weak solution of $\Delta v=0$ in $\Omega \subset \mathbb{R}^n$ and $0\leq \phi \in C^{\infty}_0(\Omega)$ is a test function. The present paper is motivated by Jaros' results. The inequality \eqref{eq_Cacci_L2} can be derived from \eqref{PiconeL_2} by integrating over $\Omega$ with $u=v\phi$, then applying the Cauchy-Schwartz inequality and the Young inequality. The $L^p$-version of inequality \eqref{eq_Cacci_L2} was obtained in \cite{IS2003}, \cite{LLM2007}, and \cite{PRS2008}. Note that the authors in \cite{RSS_VF} have obtained the weighted anisotropic Hardy and Rellich inequalities making use of the (first and second order) anisotropic Picone identities. We also refer to a recent open access book \cite{RS_book} for further discussions in this direction. Here, we establish a version of Picone's identity for general vector fields:  
\begin{lem}\label{lem_Picone}
	Let $\Omega \subset M$ be an open set. Let $p>1$. Let $u$ and $v$ be differentiable in a given set $\Omega$ with $v(x)\neq 0$ in $\Omega$. Define
	\begin{align}
	&L(u,v):= |\nabla_X u|^p -p\frac{|u|^{p-2}u}{|v|^{p-2}v} |\nabla_Xv|^{p-2}   \langle \nabla_{X} v, \nabla_{X} u\rangle     + (p-1)\frac{|u|^p}{|v|^p} |\nabla_Xv|^p, \\
	&R(u,v):= |\nabla_Xu|^p -  |\nabla_Xv|^{p-2} \langle \nabla_{X} v, \nabla_{X} \left( \frac{|u|^p}{|v|^{p-2}v} \right)\rangle.
	\end{align}
	 Then we have
	\begin{equation}
		L(u,v)=R(u,v) \geq 0. 
	\end{equation}
	Moreover, we have $L(u,v)=0$ a.e. in $\Omega$ if and only if $u=cv$ a.e. in $\Omega$ with a positive constant $c$.
\end{lem}

A Carnot group (stratified group) version of the Picone identity was obtained in \cite{RS_Green} for a general case. Also, Niu, Zhang, and Wang in \cite{NZW01} obtained the Picone identity on the Heisenberg group and remarked that it could be extended to general vector fields satisfying H\"ormander's condition. This idea was later extended in \cite[Section 11.6]{RS_book}. Indeed, all the extensions are a direct reworking of the proof \cite[Lemma 2.1]{NZW01}. Therefore, here we omit the proof of  Lemma \ref{lem_Picone}. 

The paper is organised in the following way: 
In Section \ref{Sec3}, we prove the uniqueness and simplicity of the principal frequency (or the first eigenvalue) of the Dirichlet $p$-sub-Laplacian for general vector fields.
In Section \ref{Sec4}, we obtain the Caccioppoli inequalities for general vector fields making use of Picone's identity. Then, we present some examples on the Grushin plane and the Heisenberg group.

\section{Principal frequency of $p$-sub-Laplacians}\label{Sec3}
\begin{lem}\label{lem_J}
	Assume that there exists a strictly positive sup-solution of \eqref{Dirichlet_BV}. Then we have
	\begin{equation}\label{Functional_2}
\int_{\Omega} |\nabla_X u|^p dx \geq \lambda \int_{\Omega} |u|^p dx,
\end{equation}
for all functions $u \in \mathring{S}^{1,p}(\Omega)$.
\end{lem}
We recall that a sup-solution $v$ of \eqref{Dirichlet_BV} has the following form
\begin{equation*}
\int_{\Omega}  |\nabla_X v|^{p-2} \langle \nabla_{X}v, \nabla_{X} \phi \rangle dx - \lambda \int_{\Omega}|v|^{p-2}v \phi dx \geq 0,
\end{equation*}
for all $\phi \in \mathring{S}^{1,p}(\Omega)$.
\begin{proof}[Proof of Lemma \ref{lem_J}]
	Suppose that $v$ is a strictly positive sup-solution of \eqref{Dirichlet_BV} in $\Omega$ and assume that $u \in C_0^{\infty}(\Omega)$. Then for a given small $\varepsilon>0$ in \eqref{inq_sup} we set 
	\begin{equation*}
		\phi := \frac{|u|^p}{(v+\varepsilon)^{p-1}}.
	\end{equation*}
	A straightforward computation gives that
	\begin{align*}
		\int_{\Omega} \lambda v^{p-1}\frac{|u|^p}{(v+\varepsilon)^{p-1}} dx &\leq \int_{\Omega}   |\nabla_Xv|^{p-2} \langle \nabla_{X} v, \nabla_{X} \left( \frac{|u|^p}{(v+\varepsilon)^{p-1}}  \right)\rangle  dx \\
		 =& \int_{\Omega} |\nabla_Xu|^p dx - \int_{\Omega} |\nabla_Xu|^p -  |\nabla_Xv|^{p-2} \langle \nabla_{X} v, \nabla_{X}  \left( \frac{|u|^p}{(v+\varepsilon)^{p-1}}  \right)\rangle  dx \\
	 = &\int_{\Omega} |\nabla_Xu|^p dx - \int_{\Omega} L(u,v+\varepsilon)dx. 
	\end{align*}
	In the last line, we have used the Picone identity. Now by taking the limit as $\varepsilon\rightarrow 0^+$ and making use of Fatou's lemma on the left-hand side and the Lebesgue dominated convergence theorem on the right-hand side of above expression, we arrive at
	 \begin{equation}
	 	\int_{\Omega} |\nabla_Xu|^p dx - \lambda \int_{\Omega} |u|^p dx - \int_{\Omega} L(u,v) dx \geq 0.
	 \end{equation}
	 Since $L(u,v)\geq 0$ a.e. in $\Omega$, it yields 
	 \begin{equation*}
	 	\int_{\Omega} |\nabla_Xu|^p dx - \lambda \int_{\Omega} |u|^p dx  \geq 0.
	 \end{equation*}
	 By a density argument, we assert that the claim is valid for $\mathring{S}^{1,p}(\Omega)$. 
\end{proof}

\begin{thm}\label{thm_eigenvalue}
	Let $\Omega\subset M$ be a bounded open set. If there exists $\lambda$ and a strictly positive $v \in \mathring{S}^{1,p}(\Omega)$ such that 
	\begin{equation}\label{eq_weak}
		\int_{\Omega} |\nabla_X v |^{p-2} \langle \nabla_{X} v, \nabla_{X} \phi \rangle dx \geq \lambda \int_{\Omega} |v|^{p-2}v \phi dx,  
	\end{equation}
	for every nonnegative function $\phi \in \mathring{S}^{1,p}(\Omega)$, then we have
	\begin{equation}\label{inq_eigenvalue}
		\lambda_1(\Omega) \geq \lambda,
	\end{equation}
	and 
	\begin{equation}\label{inq_J1}
		\int_{\Omega} |\nabla_X u|^p dx \geq \lambda \int_{\Omega} |u|^p dx,
	\end{equation}
	for all $u \in \mathring{S}^{1,p}(\Omega)$.
\end{thm}
Note that here and after we denote by $\lambda_1$ the principal frequency of $\mathcal{L}_p$ and by $u_1$ the corresponding positive eigenfunction without the existence argument, that is, we assume that 
 \begin{equation}\label{assumption}
  \int_{\Omega} |\nabla_X u_1 |^{p-2} \langle \nabla_{X} u_1, \nabla_{X} \phi\rangle dx = \lambda_1(\Omega) \int_{\Omega} |u_1|^{p-2}u_1 \phi dx, 
 \end{equation}
 for every $\phi \in \mathring{S}^{1,p}(\Omega)$. 
 
\begin{proof}[Proof of Theorem \ref{thm_eigenvalue}]
	Note that inequality \eqref{inq_J1} was proved in Lemma \ref{lem_J}. So now we need to prove \eqref{inq_eigenvalue}. 
	
	Let $u_1 \in \mathring{S}^{1,p}(\Omega)$ be the eigenfunction associated to the principal frequency $\lambda_1(\Omega)$. By choosing  $\varepsilon >0$, we integrate $L(u_1,v+\varepsilon)$ over $\Omega$, then using  \eqref{eq_weak} and \eqref{assumption} with $\phi=u_1$, we have
	\begin{align*}
		0 \leq& \int_{\Omega} L(u_1,v+\varepsilon) dx \\
		 =&\int_{\Omega} |\nabla_Xu_1|^pdx  - \int_{\Omega}|\nabla_Xv|^{p-2} \langle \nabla_{X} v, \nabla_{X} \left( \frac{|u_1|^p}{(v+\varepsilon)^{p-1}}\right) \rangle dx \\
		 \leq & \,\lambda_1(\Omega) \int_{\Omega} |u_1|^p dx - \lambda \int_{\Omega} v^{p-1} \frac{|u_1|^p}{(v+\varepsilon)^{p-1}} dx.
	\end{align*}
	Now we take taking the limit as $\varepsilon \rightarrow 0^+$, we arrive at 
	\begin{equation*}
		0 \leq \lambda_1(\Omega) - \lambda,
	\end{equation*}
	which proves Theorem \ref{thm_eigenvalue}.
\end{proof}
\begin{cor}[Uniqueness]\label{uniq}
	 If there exists $\lambda$ and a strictly positive $v \in \mathring{S}^{1,p}(\Omega)$ such that 
	\begin{equation}\label{eq=}
	\int_{\Omega} |\nabla_X v |^{p-2} \langle \nabla_{X} v, \nabla_{X} \phi\rangle dx = \lambda \int_{\Omega} |v|^{p-2}v \phi dx,  
	\end{equation}
	for every nonnegative function $\phi \in \mathring{S}^{1,p}(\Omega)$, then we have 
	\begin{equation}
		\lambda_1(\Omega)=\lambda.
	\end{equation}
\end{cor}
\begin{proof}[Proof of Corollary \ref{uniq}]
	Let $u_1 \in \mathring{S}^{1,p}(\Omega)$ be the eigenfunction associated to the principal frequency $\lambda_1(\Omega)$. By choosing  $\varepsilon >0$, we integrate $L(u_1,v+\varepsilon)$ over $\Omega$, then using \eqref{assumption} and \eqref{eq=} with $\phi=v$, we have
	\begin{align*}
	0 \leq& \int_{\Omega} L(v,u_1+\varepsilon) dx \\
	=&\int_{\Omega} |\nabla_Xv|^pdx  - \int_{\Omega}|\nabla_Xu_1|^{p-2} \langle \nabla_{X} u_1, \nabla_{X} \left( \frac{|v|^p}{(u_1+\varepsilon)^{p-1}}\right)\rangle dx \\
	=& \lambda \int_{\Omega} |v|^p dx - \lambda_1(\Omega) \int_{\Omega} u_1^{p-1} \frac{|v|^p}{(u_1+\varepsilon)^{p-1}} dx.
	\end{align*}
	Now we take the limit as $\varepsilon \rightarrow 0^+$, so we get 
	$	\lambda_1(\Omega) \leq \lambda$. On the other hand, by Theorem \ref{thm_eigenvalue} we have the fact that $\lambda_1(\Omega) \geq \lambda$, that is, we arrive at 
	\begin{equation*}
	\lambda_1(\Omega) = \lambda,
	\end{equation*}
	completing the proof.
\end{proof}

\begin{thm}[Simplicity]\label{simplicity}
Any eigenfunction $u \in \mathring{S}^{1,p}(\Omega)$ of $\mathcal{L}_p$ associated to the principal frequency $\lambda_1(\Omega)$ is a constant multiple of $u_1$.
\end{thm}
\begin{proof}[Proof of Theorem \ref{simplicity}]
	Let both $u$ and $u_1$ be eigenfunctions associated to the principal frequency $\lambda_1$. Let $L(u,u_1)\geq 0$ in $\Omega$. For $\varepsilon >0$, we integrate $L(u,u_1+\varepsilon)$ over $\Omega$ with $v=u_1+\varepsilon$, we get
	\begin{align*}
		\int_{\Omega} L(u,u_1+\varepsilon) dx = \int_{\Omega} |\nabla_X u|^p dx - \lambda_1(\Omega) \int_{\Omega}u_1^{p-1}\frac{|u|^p}{(u_1+\varepsilon)^{p-1}}dx.
	\end{align*}
	Taking the limit as $\varepsilon \rightarrow 0^+$ and using the Lebesgue dominated convergence theorem and Fatou's lemma, we arrive at
	\begin{equation*}
			\int_{\Omega} L(u,u_1) dx \leq \int_{\Omega} |\nabla_Xu|^p dx - \lambda_1(\Omega)\int_{\Omega}|u|^p dx =0.
	\end{equation*}
	Since $ L(u,u_1)\geq$ in $\Omega$, it follows  that $L(u,u_1)=0$ a.e. in $\Omega$. So it follows from Lemma \ref{lem_Picone} that $u=Cu_1$ with $C>0$.
\end{proof}

Then, by using the Picone identity, we show monotonicity of the principal frequency  $\lambda_1(\Omega)$ in Theorem \ref{thm_eigenvalue} as the function of the set $\Omega$. 
\begin{thm}\label{thm_mono}
	Let $\Omega$ and $\widetilde{\Omega}$ be subsets of $M$. If $\Omega \subset \widetilde{\Omega}$, then we have
	\begin{equation}
		\lambda_1(\Omega) \geq \lambda_1(\widetilde{\Omega}).
	\end{equation}
\end{thm}
\begin{proof}[Proof of Theorem \ref{thm_mono}]
	We have $\Omega \subset \widetilde{\Omega}$, and let $u_1$ and $\widetilde{u}_{1}$ be positive eigenfunctions associated to $\lambda_1(\Omega)$ and $\lambda_1(\widetilde{\Omega})$, respectively (see the assumption \eqref{assumption}). For $\varepsilon>0$ we consider $L(u_1,\widetilde{u}_{1}+\varepsilon)$ where $u=u_1$ and $v=\widetilde{u}_{1}+\varepsilon$. By using inequality \eqref{eq_weak}, we get
	\begin{align*}
		\int_{\Omega} L(u_1,\widetilde{u}_{1}+\varepsilon) dx &= \int_{\Omega}|\nabla_Xu_1|^p dx - \int_{\Omega}  |\nabla_X \widetilde{u}_{1}|^{p-2} \langle\nabla_{X} \widetilde{u}_{1}, \nabla_{X}\left( \frac{u_1^p}{(\widetilde{u}_{1}+\varepsilon)^{p-1}}\right) \rangle dx \\
		& \leq \lambda_1(\Omega)\int_{\Omega} u_1^p dx - \lambda_1(\widetilde{\Omega})\int_{\Omega} \widetilde{u}_{1}^{p-2}\frac{u_1^p}{(\widetilde{u}_{1}+\varepsilon)^{p-1}} dx.
	\end{align*}
	Now by taking the limit as $\varepsilon \rightarrow 0^+$ and making use of Fatou's lemma and the Lebesgue dominated convergence theorem, we obtain
	\begin{equation}
(\lambda_1(\Omega)-\lambda_1(\widetilde{\Omega})) \int_{\Omega} u_1^p dx \geq 	\int_{\Omega} L(u_1,\widetilde{u}_{1}) dx \geq 0,
	\end{equation}
	which yields 
	\begin{equation}
		\lambda_1(\Omega) \geq \lambda_1(\widetilde{\Omega}).
	\end{equation}
	This proves Theorem \ref{thm_mono}.
\end{proof}
\section{Caccioppoli inequalities for general vector fields}\label{Sec4}
Here we give the Caccioppoli inequalities for general vector fields by making use of Picone's identity, and we discuss some of their corollaries on the Grushin plane and the Heisenberg group.
\begin{thm}[Caccioppoli inequality]\label{thm_Cacci}
Let $v$ be a positive sub-solution of \eqref{Dirichlet_BV} in $\Omega \subset M$. Then for every fixed $q>p-1$, $1<p<\infty$, and $\lambda \in \mathbb{R}$ we have 
	 \begin{equation}\label{Caccioppoli}
	 \int_{\Omega}v^{q-p}\phi^p|\nabla_X v|^pdx \leq \left( \frac{p}{q-p+1} \right)^p \int_{\Omega} v^q |\nabla_X \phi|^p dx + \frac{\lambda p}{q-p+1}\int_{\Omega}v^{q}\phi^p dx,
	 \end{equation} 
	 for all nonnegative functions $\phi \in C^{\infty}_0(\Omega)$.
\end{thm}

	Note that in the case when $q=p$ and $\lambda=0$ in \eqref{Caccioppoli} we have
	\begin{equation}
		\int_{\Omega} \phi^p |\nabla_X v|^p dx \leq p^p \int_{\Omega} v^p |\nabla_X \phi|^p dx.
	\end{equation}

\begin{proof}[Proof of Theorem \ref{thm_Cacci}]
 Setting $u:=v^{q/p}\phi$ in $L(u,v)$ we have
	\begin{align*}
		\int_{\Omega}L(v^{q/p}\phi,v)dx = &\int_{\Omega} |\nabla_X (v^{q/p}\phi)|^p dx + (p-1)\int_{\Omega}v^{q-p}|\phi \nabla_X v|^p dx \\
		& - p\int_{\Omega} |v^{\frac{q-p}{p}}\phi|^{p-1} |\nabla_X v|^{p-2} \langle \nabla_{X} v, \nabla_{X} (v^{q/p}\phi) \rangle dx\\
		=& \int_{\Omega} |\nabla_X (v^{q/p}\phi)|^p dx - (q-p+1)\int_{\Omega}v^{q-p}|\phi \nabla_X v|^p dx \\
		& - p\int_{\Omega} |v^{\frac{q-p}{p}}\phi|^{p-1}  |\nabla_Xv|^{p-2}v^{q/p}  \langle \nabla_{X} v, \nabla_{X} \phi\rangle dx.
	\end{align*}
	In the last line, we have used the equality $	\nabla_X (v^{q/p}\phi) = \frac{q}{p} v^{\frac{q-p}{p}} \phi \nabla_Xv + v^{q/p}\nabla_X \phi$. Noting that 
	\begin{align*}
	\langle \nabla_X v, \nabla_X \phi\rangle &= \sum_{i=1}^{N}X_iv X_i\phi \leq  \sum_{i=1}^{N} |X_iv||X_i\phi| \\
	& \leq \left(  \sum_{i=1}^{N}(X_i v)^2 \right)^{1/2} \left(  \sum_{i=1}^{N} (X_i \phi)^2\right)^{1/2} = |\nabla_{X} v| |\nabla_{X} \phi|,
	\end{align*}
 and recalling Young's inequality in the form
	\begin{equation*}
		ab^{p-1} \leq \frac{a^p}{ps^{p-1}} + \frac{p-1}{p}sb^p, \,\, a,b\geq 0 \, \text{and} \,\, s >0,
	\end{equation*}
with $a= v^{q/p}|\nabla_X \phi|$ and $b= v^{\frac{q-p}{p}}\phi|\nabla_X v|$, we obtain
	\begin{align*}
			0\leq \int_{\Omega}L(v^{q/p}\phi,v) dx= &\int_{\Omega} |\nabla_X (v^{q/p}\phi)|^p dx  + s^{1-p}\int_{\Omega} v^q |\nabla_X \phi|^p dx\\
			& - (q-p+1-s(p-1))\int_{\Omega}v^{q-p}|\phi \nabla_X v|^pdx \\
			\leq & \lambda \int_{\Omega}|v^{q/p }\phi|^p dx + s^{1-p}\int_{\Omega} v^q |\nabla_X \phi|^p dx \\
			& - (q-p+1-s(p-1))\int_{\Omega}v^{q-p}|\phi \nabla_X v|^pdx.
	\end{align*}
Here we have used $\int_{\Omega} |\nabla_X u|^p dx \leq \int_{\Omega}\lambda |u|^pdx$ for the sub-solution of \eqref{Dirichlet_BV}. Thus, we arrive at 
\begin{align}
	\int_{\Omega}v^{q-p}\phi^p|\nabla_X v|^pdx \leq& \frac{s^{1-p}}{ q-p+1-s(p-1)}\int_{\Omega} v^q |\nabla_X \phi|^p dx \\
	&+ \frac{\lambda}{q-p+1-s(p-1)}  \int_{\Omega}v^{q}\phi^p dx\nonumber.
\end{align}	
Now we choose the suitable constant as $s= \frac{q-p+1}{p}$ which leads to
\begin{equation}
	\int_{\Omega}v^{q-p}\phi^p|\nabla_X v|^pdx \leq \left( \frac{p}{q-p+1} \right)^p \int_{\Omega} v^q |\nabla_X \phi|^p dx + \frac{\lambda p}{q-p+1}\int_{\Omega}v^{q}\phi^p dx.
\end{equation} 
This proves Theorem \ref{thm_Cacci}.
\end{proof}
Let us discuss some particular cases of the Caccioppoli inequalities on the Grushin plane and the Heisenberg group. One of the important examples of a sub-Riemannian manifold is the Grushin plane. Recall that the Grushin plane $G$ is the space $\mathbb{R}^2$ with vector fields 
\begin{equation*}
X_1 = \frac{\partial}{\partial x_1}, \,\, \text{and} \,\, X_2 = x_1\frac{\partial}{\partial x_2},
\end{equation*}
and the gradient
\begin{equation}
	\nabla_G = (X_1,X_2)
\end{equation}
for $x:=(x_1,x_2)\in \mathbb{R}^2$.
\begin{cor}
	Let $\Omega \subset G$. Let $v$ be a positive sub-solution of 
	\begin{equation}\label{Dirichlet_BV_Grush}
	\begin{cases}
		- \nabla_{G}\cdot(|\nabla_{G}u|^{p-2}\nabla_{G} u) = \lambda |u|^{p-2}u, \,\,& \text{in} \,\, \Omega,	\\
	u=0,\,\, & \text{on} \,\, \partial \Omega.  
	\end{cases}
	\end{equation}
	Then for every fixed $q>p-1$ and $1<p<\infty$ we have 
	\begin{equation}
	\int_{\Omega}v^{q-p}\phi^p|\nabla_G v|^pdx \leq \left( \frac{p}{q-p+1} \right)^p \int_{\Omega} v^q |\nabla_G \phi|^p dx + \frac{\lambda p}{q-p+1}\int_{\Omega}v^{q}\phi^p dx,
	\end{equation} 
	and for $q=p$ and $\lambda=0$ we have
	\begin{equation}
	\int_{\Omega} \phi^p |\nabla_G v|^p dx \leq p^p \int_{\Omega} v^p |\nabla_G \phi|^p dx,
	\end{equation}
	for all nonnegative functions $\phi \in C^{\infty}_0(\Omega)$.
\end{cor}

Let $\mathbb{H}^n$ be the Heisenberg group, that is,  $\mathbb{R}^{2n+1}$ equipped with the group law 
\begin{equation*}
\xi \circ \xi' := (x + x', y + y', t + t'+2(x \cdot y' - x' \cdot y)),
\end{equation*}
where $\xi:= (x,y,t) \in \mathbb{R}^{2n+1}$ with $x\in \mathbb{R}^n$, $y\in \mathbb{R}^n$, $t\in \mathbb{R}$. The dilation operation on the Heisenberg group with respect to the group law has the form
\begin{equation*}
\delta_{\lambda}(\xi) := (\lambda x, \lambda y, \lambda^2 t) \,\, \text{for}\,\, \lambda>0.
\end{equation*} 
The Heisenberg group is a basic example of step 2 stratified Lie groups (Carnot groups).

The Lie algebra $\mathfrak{h}$ of the left-invariant vector fields on the Heisenberg group $\mathbb{H}^n$ is spanned by 
\begin{equation*}
X_j:= \frac{\partial }{\partial x_j} + 2y_j\frac{\partial }{\partial t},
\end{equation*}
\begin{equation*}
Y_j:= \frac{\partial }{\partial y_j} - 2x_j\frac{\partial }{\partial t},
\end{equation*}
with their (non-zero) commutator
\begin{equation*}
T=[X_j,Y_j]= - 4 \frac{\partial}{\partial t}.
\end{equation*} 
The horizontal gradient on $\mathbb{H}^n$ is given by 
\begin{equation*}
\nabla_{H}:= (X_1,\ldots,X_n,Y_1,\ldots,Y_n),
\end{equation*} 
so the sub-Laplacian on $\mathbb{H}^n$ is given by
\begin{equation*}
\mathcal{L}:=\sum_{j=1}^{n}X_j^2 + Y_j^2. 
\end{equation*} 

Note that in the case of the general stratified Lie groups $\mathbb{G}$ (Carnot groups) we have $X_k = - X^*_k$, so the formula \eqref{Lp} can be rewritten as 
\begin{equation}
\mathcal{L}_p f:=	\nabla_{H}\cdot(|\nabla_{H}f|^{p-2}\nabla_{H} f),
\end{equation}
which is called the $p$-sub-Laplacian.

\begin{cor}
	Let $\Omega \subset \mathbb{H}^n$. Let $v$ be a positive sub-solution of 
	\begin{equation}\label{Dirichlet_BV_Heis}
	\begin{cases}
		- \nabla_{H}\cdot(|\nabla_{H}u|^{p-2}\nabla_{H} u) = \lambda |u|^{p-2}u, \,\,& \text{in} \,\, \Omega,	\\
	u=0,\,\, & \text{on} \,\, \partial \Omega.  
	\end{cases}
	\end{equation} 
	Then for every fixed $q>p-1$ and $1<p<\infty$ we have 
	\begin{equation}
	\int_{\Omega}v^{q-p}\phi^p|\nabla_H v|^pdx \leq \left( \frac{p}{q-p+1} \right)^p \int_{\Omega} v^q |\nabla_H \phi|^p dx + \frac{\lambda p}{q-p+1}\int_{\Omega}v^{q}\phi^p dx,
	\end{equation} 
	and for $q=p$ and $\lambda=0$ we have
	\begin{equation}
	\int_{\Omega} \phi^p |\nabla_H v|^p dx \leq p^p \int_{\Omega} v^p |\nabla_H\phi|^p dx,
	\end{equation}
	for all nonnegative functions $\phi \in C^{\infty}_0(\Omega)$.
\end{cor}
Here the Caccioppoli inequality on the Heisenberg group seems to be also new.


\begin{thebibliography}{NZW01}
	\bibitem{AH98}
	Allegretto W., Huang Y.X.: 
	\newblock A Picone's identity for the $p$-Laplacian and applications. 
	\newblock \textit{Nonlinear Anal.}, 32, 819-830 (1998)
	\bibitem{Bony}
	Bony J.M.: 
	\newblock Principe du maximum, int\'egalit\'e de Harnack et unicit\'e du probl\'eme de Cauchy pour les op\'erateurs elliptique d\'eg\'en\'er\'es. \textit{Ann. Inst. Fourier Grenoble}, 119(1), 277-204 (1969)
	\bibitem{CDG1993}
Capogna L., Danielli D., and Garofalo N.:
\newblock An embedding theorem and the Harnack inequality for nonlinear subelliptic equation. {\it Comm. Partial Differential Equations}, 18, 1765--1794 (1993) 
\bibitem{CG98}
Capogna L., Garofalo N.:
\newblock Boundary behavior of non-negative solutions of subelliptic equations in NTA domains for Carnot-Carath\'eodory metrics. \textit{J. Fourier Anal. Appl.}, 4-5(4), 403-432 (1998)
\bibitem{CGN02}
Capogna L., Garofalo N., Nhieu D.M.:
\newblock Properties of harmonic measures in the Dirichlet problem for nilpotent Lie groups of Heisenberg type. \textit{Am. J. Math.}, 124(2), 273-306 (2002)

	\bibitem{Danielli95}
	Danielli D.: 
	\newblock Regularity at the boundary for solutions of nonlinear subelliptic equations. \textit{Indiana Univ. Math. J.}, 44, 269--286 (1995)
	
	\bibitem{DGMN}
	Danielli D., Garofalo D., Munive I., Nhieu D.M.: The Neuman problem on the Heisenberg group, regularity and boundary behavior of solutions, in preparation  

	
	\bibitem{GV00}
	Garofalo N.,Vassilev D.: 
	\newblock Regularity near the characteristic set in the non-linear Dirichlet problem and conformal geometry of sub-Laplacians on Carnot Groups. \textit{Math. Ann.}, 318, 453-516 (2000)  
	\bibitem{Gaveau}
	Gaveau B.: 
	\newblock Principe de moindre action, propagation de la chaleur at estim'ees sous elliptiques sur certain groupes nilpotents. \textit{Acta Math.}, 139, 95--153 (1977)
	
	\bibitem{IS2003}
Iwaniec T., Sbordone C.:
\newblock Caccioppoli estimates and very weak solutions of elliptic equations. \textit{Atti Accad. Naz. Lincei CI. Sci. Fis. Mat. Natur. Rend. Lincei (9) Mat. Appl.} 14(3), 189-205 (2004). Renato Caccioppoli and modern analysis	
\bibitem{Jaros_A-horm}
Jaros J.:
\newblock A-harmonic Picone's identity with applications.
\newblock \textit{Annali di Matematica}, 194(3), 719--729 (2015)

\bibitem{Jaros_Cacci}
Jaros J.:
\newblock Caccioppoli estimates through an anisotropic Picone's identity. 
\newblock \textit{Proc. Am. Math. Soc.}, 143, 1137--1144  (2015)
\bibitem{Jer_1}
Jerison D.S.: 
\newblock The Dirichlet problem for the Kohn Laplacian on the Heisenberg group. I. \textit{J. Funct. Anal.}, 43(1), 97-142 (1981)
\bibitem{Jer_2}
Jerison D.S.: 
\newblock The Dirichlet problem for the Kohn Laplacian on the Heisenberg group. II. \textit{J. Funct. Anal.}, 43(2), 224-257 (1981)
\bibitem{Kohn-Nirenberg}
Kohn J.J., Nirenberg L.:
\newblock Non-coercive boundary value problems. \textit{Commun. Pure Appl. Math.}, 18, 443--492 (1965)	
	\bibitem{Lindqvist}
	Lindqvist P.:
	\newblock On the equation ${\rm div}(|\nabla u|^{p-2}\nabla u) + \lambda |u|^{p-2}u=0.$ 
	\newblock \textit{Proc. Am. Math. Soc.}, 109(1), 157--164 (1990)
	
	\bibitem{LLM2007}
	Liskevich V., Lyakhova S., Moroz V.:
	\newblock Positive solutions to nonlinear $p$-Laplace equations with Hardy potential in exterior domains. 
	\newblock {\it J. Differential Equations}, 232(1), 212--252 (2007)
	
	\bibitem{NZW01}
Niu P., Zhang H. and Wang Y.: Hardy Type and Rellich Type Inequalities on the Heisenberg Group. 
\newblock \textit{Proc. Am. Math. Soc.}, 129(12), 3623--3630 (2001) 

\bibitem{Niu99}
Pengcheng N.: Nonexistence for semiliniear equations and systems in the Heisenberg group. \textit{Journal of Mathematical Analysis and Applications}. 240, 47-59 (1999)
\bibitem{PRS2008}
Pigalo S., Rigoli M., and Setti A.G.:
\newblock Vanishing and finiteness results in geometric analysis: A generalization of the Bochner technique, \textit{Progress in Mathematics}, vol. 266, Birkh\"auser Verlag, Basel, 2008.


	\bibitem{RSS_VF}
	Ruzhansky M., Sabitbek B., Suragan D.: 
	\newblock Weighted anisotropic Hardy and Rellich type inequalities for general vector fields.
	\newblock \textit{Nonlinear Differ. Equ. Appl.} 26, 13 (2019). https://doi.org/10.1007/s00030-019-0559-5
	\bibitem{RS_book}
	Ruzhansky M., Suragan D.:
	\newblock Hardy inequalities on homogeneous groups.
	\newblock \textit{Progress in Math.} Vol. 327, Birkh\"auser, 588 pp, (2019)
\bibitem{RS_Green}
	Ruzhansky M., Suragan D.:
	\newblock Green's identities, comparison principle and uniqueness of positive solutions for nonlinear $p$-sub-Laplacian equations on stratified Lie groups.
	\newblock \textit{Potential Anal} (2019). https://doi.org/10.1007/s11118-019-09782-y


\bibitem{Swan}
Swanson C.A.:
\newblock Picone's identity. \textit{Rend. Mat.}, 8, 373--397 (1975)

\bibitem{YZP05}
Yazhou H., Xuebo L., Pengcheng N.: Liouville type theorems of semilinear equations with square sum of vector fields. \textit{J. Partial Diff. Eqs.}, 18, 149-153 (2005)	
	
\end{thebibliography}
\end{document}